\documentclass[a4paper,12pt,draft]{amsart}
\usepackage[margin=30mm]{geometry}
\usepackage{amssymb, amsmath, amsthm, amscd}

\usepackage[T1]{fontenc}
\usepackage{eucal,mathrsfs,dsfont}%gives nicer set-names. Use \mathds instead of \mathds
\usepackage{color}%cyan,red;magenta,green;yellow,blue
\definecolor{mno}{rgb}{0.5,0.1,0.5}

\newcommand{\R}{\mathds R}
\newcommand{\Rd}{{\mathds R^d}}
\newcommand{\Pp}{\mathds P}
\newcommand{\Ee}{\mathds E}
\newcommand{\I}{\mathds 1}

\newcommand{\var}{\textmd{Var}}
\newcommand{\Bb}{\mathscr{B}}

\newtheorem{theorem}{Theorem}[section]

\newtheorem{proposition}[theorem]{Proposition}

\theoremstyle{definition}

\newtheorem{example}[theorem]{Example}
\newtheorem{remark}[theorem]{Remark}

\begin{document}

\title[Constructions of Coupling Processes for L\'{e}vy Processes]{\bfseries Constructions of Coupling Processes for L\'{e}vy Processes}

\author{Bj\"{o}rn\ B\"{o}ttcher\,\,\,\,\,\, Ren\'{e} L.\ Schilling\,\,\,\,\,\, Jian Wang}
\thanks{\emph{B.\ B\"{o}ttcher:} TU Dresden, Institut f\"{u}r Mathematische Stochastik, 01062 Dresden, Germany. \texttt{bjoern.boettcher@tu-dresden.de}}
\thanks{\emph{R.\ Schilling:} TU Dresden, Institut f\"{u}r Mathematische Stochastik, 01062 Dresden, Germany. \texttt{rene.schilling@tu-dresden.de}}
\thanks{\emph{J.\ Wang:}
School of Mathematics and Computer Science, Fujian Normal
University, 350007, Fuzhou, P.R. China \emph{and} TU Dresden,
Institut f\"{u}r Mathematische Stochastik, 01062 Dresden, Germany.
\texttt{jianwang@fjnu.edu.cn}}

\date{}

\maketitle

\begin{abstract}We
construct optimal Markov couplings of L\'{e}vy processes, whose
L\'evy (jump) measure has an absolutely continuous component. The construction is based on properties of subordinate
Brownian motions and the coupling of Brownian motions by reflection.

\medskip


\medskip

\end{abstract}

\section{Introduction and Main Results}\label{section1}
It is well known that a L\'evy process $(X_t)_{t\geq 0}$ on $\Rd$
can be decomposed into three independent parts, i.e.\ the Gaussian
part, the drift part and the jump part. The corresponding symbol or characteristic exponent
(see \cite{JBOOK, SA}) of $X_t$ is given by
$$
    \psi(\xi)= \frac{1}{2}\langle Q\xi, \xi\rangle +i\langle b,\xi\rangle + \int_{z\neq 0}\!\! \Big( 1-e^{-i\langle \xi,z\rangle}+i\langle \xi,z\rangle\I_{\{|z|\le 1\}}\Big)\nu(dz),
$$
where $Q=(q_{j,k})_{j,k=1}^d$ is a positive semi-definite matrix,
$b\in\Rd$ is the drift vector and $\nu$ is the L\'evy or jump
measure; the L\'{e}vy measure $\nu$ is a $\sigma$-finite measure on
$\R^d\setminus\{0\}$ such that $\int_{z\neq 0}(1\wedge |z|^2)\nu(dz)<\infty$. If
the matrix $Q$ is strictly positive definite, regularity properties
for the semigroup of a L\'{e}vy process can be easily derived from
that of Brownian motion. However, when a L\'{e}vy process only has a
pure jump part (i.e.\ $Q=0$ and $\nu\neq0$), the situation is
completely different and, in general, more difficult to deal with. As a
continuation of our recent work \cite{RWJ}, we aim to construct
optimal Markov coupling processes of L\'{e}vy process $X_t$, by
assuming that the corresponding L\'{e}vy measure has absolutely
continuous lower bounds.

It has been proven in \cite[Theorem 3.1]{W1} and \cite[Theorem
1.1]{RWJ} that under some mild conditions compound Poisson processes admit successful
couplings, and the corresponding transition probability function satisfies
\begin{equation}\label{1proff111}\|P_t(x,\cdot)-P_t(y,\cdot)\|_{\var}\le
\frac{C(1+|x-y|)}{\sqrt{t}}\wedge 2\qquad\textrm{ for }t>0\textrm{
and }x, y\in\R^d,\end{equation} where
$\|\mu\|_{\var}$ denotes the total variation norm of the signed
measure $\mu$; moreover, the factor $\sqrt{t^{-1}}$ in
the inequality \eqref{1proff111} is sharp for $t>0$ large enough. The following question is natural: \emph{Is
the rate $\sqrt{t^{-1}}$ also optimal for general L\'{e}vy processes
that possess the coupling property}? Note that the L\'{e}vy measure
$\nu$ is always finite outside a neighborhood of $0$. Thus, the
behavior of $\nu$ around the origin will be crucial for optimal
estimates of $\|P_t(x,\cdot)-P_t(y,\cdot)\|_{\var}$ as $t$ tends to
infinity.

Before stating our main results, we first present some necessary
notations. A nonnegative function $f$ on $(0,\infty)$ is called a
\textit{Bernstein function} if $f\in C^\infty(0,\infty)$, $f\ge0$
and for all $k\ge1$, $(-1)^kf^{(k)}(x)\le 0.$ Any Bernstein function
$f$ has a L\'{e}vy-Khintchine representation
\begin{equation}\label{1proff0}f(\lambda)=a+b\lambda+\int_0^\infty(1-e^{-\lambda s})\mu(ds),\qquad \lambda>0,\end{equation} where $a$, $b>0$ and $\mu$
is a Radon measure on $(0,\infty)$ such that
$\int_0^\infty(s\wedge1)\mu(ds)<\infty.$  In particular, the
L\'{e}vy triplet $(a,b,\mu)$ determines the Bernstein function $f$
uniquely and vice versa, e.g.\ see \cite[Theorem 3.2]{RSW}.

\begin{theorem}\label{section1th1} Let $X_t$ be a L\'{e}vy process on $\R^d$ and $\nu$ be its L\'{e}vy measure. Assume that
\begin{equation}\label{1proff2}\nu(dz)\ge  |z|^{-d}f(|z|^{-2})dz,\end{equation} where
$f$ is a Bernstein function. Then, there is a constant $C>0$ such
that for $x$, $y\in\R^d$ and $t>0$,
\begin{equation}\label{1proff1}\|P_t(x,\cdot)-P_t(y,\cdot)\|_{\var}\le \bigg(\frac{|x-y|}{\sqrt{2}\pi}\int_0^\infty \frac{1}{\sqrt{r}}e^{-ctf(r)}dr\bigg)\wedge \frac{C(1+|x-y|)}{\sqrt{t}}\wedge 2,\end{equation}
 where  $c=\pi^{d/2}\cos 1\big/ (2d\Gamma(d/2+1))$. \end{theorem}

\medskip

Since $\int \big (1\wedge|z|^2)\nu(dz)<\infty$, we have $\int
(1\wedge|z|^2) |z|^{-d}f(|z|^{-2})dz<\infty$. That is,
$$\int_0^1\frac{f(r)}{r}dr+\int_1^\infty \frac{f(r)}{r^2}dr<\infty,$$ which implies that the Bernstein function
$f$ in \eqref{1proff2} should be without drift and killing terms
(i.e.\ in the representation \eqref{1proff0} we have $a=b=0$). Based
on the coupling of random walks, we proved in \cite[Corollaries 4.2
and 4.4]{RWJ} that any L\'{e}vy process, which is either strong
Feller or whose L\'{e}vy measure has an absolutely continuous
component, has the coupling property and \eqref{1proff111} holds.
Thus, \eqref{1proff2} yields that \eqref{1proff111} is valid in our
setting. That is, the key and novel statement of Theorem
\ref{section1th1} is the first term on the right hand side of the
estimate \eqref{1proff1}. Note that for any $x$, $y\in\R^d$ and
$t\ge0$, $\|P_t(x,\cdot)-P_t(y,\cdot)\|_{\var}\le 2$, and
$\|P_t(x,\cdot)-P_t(y,\cdot)\|_{\var}$ is decreasing with respect to
$t$. Hence the asymptotic of
 $\|P_t(x,\cdot)-P_t(y,\cdot)\|_{\var}$ as $t\rightarrow\infty$
 is more interesting. Obviously, $\int_0^\infty
\frac{1}{\sqrt{r}}e^{-ctf(r)}dr<\infty$ for some constant $c>0$ and
$t>0$ large enough if
$\liminf\limits_{r\rightarrow\infty}\frac{f(r)}{\log r}>0.$ Indeed,
we have

\begin{proposition}\label{propo} Assume that condition \eqref{1proff2} holds. Then, for any
$x$, $y\in\R^d$, as $t\rightarrow\infty$,
\begin{equation}\label{propo1}\|P_t(x,\cdot)-P_t(y,\cdot)\|_{\var}= \begin{cases} \mathsf{O}\Big(\int_0^\infty
\frac{1}{\sqrt{r}}e^{-ctf(r)}dr\Big),& f'(0+)=\infty;\\
\qquad\mathsf{O}(t^{-1/2}),\quad& f'(0+)<\infty.\end{cases}
\end{equation}

 \end{proposition}

 We will see from the next section that the assertion \eqref{propo1}
is sharp in many situations. Here we only present a typical example
to show the efficiency of Theorem \ref{section1th1}.

\begin{example}\label{th13}Assume that the L\'{e}vy measure satisfies
 $$\nu(dz)\ge c|z|^{-d-\alpha}dz$$ for $c>0$ and $\alpha\in(0,2)$. Then by Theorem \ref{section1th1}, for the associated L\'{e}vy process $X_t$, there exists a constant $C>0$ such that for any $x$,
 $y\in\R^d$ and $t>0$,
$$\|P_t(x,\cdot)-P_t(y,\cdot)\|_{\var}\le
\frac{C|x-y|}{t^{1/\alpha}}.$$
   \end{example}

\medskip

The main idea of the proof of Theorem \ref{section1th1} is to
construct coupling processes of subordinate Brownian motions, by
making full use of the coupling of Brownian motions by reflection. We will see in the next section that tail estimates for the
coupling time of those coupling processes heavily depend on the
decay of the associated Bernstein function $f(\lambda)$ as $\lambda\rightarrow0$. A number
of examples are also presented to illustrate the optimality of such
coupling processes for subordinate Brownian motions. The proofs and
some comments of Theorem \ref{section1th1} and Proposition
\ref{propo} are given in Section \ref{section3}.

\section{Couplings of Subordinate Brownian
Motions}\label{section2} In this section, we will study the
coupling property of a class of special but important L\'{e}vy
processes---subordinate Brownian motions. Examples of subordinate
Brownian motions include rotationally invariant stable L\'{e}vy processes,
relativistic stable L\'{e}vy processes and so on.

Suppose that $(B_t)_{t\ge0}$ is a Brownian motion on $\R^d$ with
$$\Ee\Big[e^{i\xi(B_t-B_0)}\Big]=e^{-t|\xi|^2},\qquad \xi\in\R^d,
t>0,$$ and $(S_t)_{t\ge0}$ is a subordinator (that is,
$(S_t)_{t\ge0}$ is a nonnegative L\'{e}vy process such that $S_t$ is
increasing and right-continuous in $t$ with $S_0=0$) independent of
$(B_t)_{t\ge0}$. For any $t\ge0$, let $\mu_t^S$ be the transition
probabilities of the subordinator $S$, i.e.\ $\mu_t^S(B)=\Pp(S_t\in
B)$ for any $B\in\Bb([0,\infty))$. It is well known that the
associated Laplace transformation of $\mu_t^S$ is given by
$$\int_0^\infty e^{-\lambda s}\mu_t^S(ds)=e^{-tf(\lambda)},\qquad \lambda>0,$$
where $f(\lambda)$ is a Bernstein function. We refer to \cite{RSW} for
more details about Bernstein functions and subordinators. Any
subordinate Brownian motion $(X_t)_{t\ge 0}$ defined by
$X_t=B_{S_t}$ is a symmetric L\'{e}vy process with
$$\Ee\Big[e^{i\xi(X_t-X_0)}\Big]=e^{-tf(|\xi|^2)},\qquad \xi\in\R^d,
t>0.$$ That is, the symbol or characteristic exponent of subordinate Brownian motion $X_t$ is
$f(|\xi|^2)$, see \cite{JBOOK}.

\medskip

Recall that the pair $(X_t,X'_t)$ is said to be a \emph{coupling of the Markov
process} $X_t$, if $(X'_t)_{t\geq 0}$ is a Markov process such that
it has same transition distribution as $(X_t)_{t\geq 0}$ but
possibly different initial distributions. In this case, $X_t$ and $X'_t$ are called the
\emph{marginal processes} of the coupling process, and the coupling time is
defined by $T:=\inf\{t\ge0: X_t=X'_t\}.$ The coupling $(X_t,X'_t)$ is
called \emph{successful} if $T$ is finite. A Markov process is said
to have successful couplings (or to have the coupling property) if for any two
initial distributions $\mu_1$ and $\mu_2$, there exists a successful
coupling with marginal processes starting from $\mu_1$ and $\mu_2$
respectively. In particular, according to \cite{Li} and the proof of
\cite[Theorem 4.1]{RWJ}, the coupling property is equivalent to the
statement that:
\begin{equation*}\label{prex1}\lim_{t\rightarrow\infty}\|P_t(x,\cdot)-P_t(y,\cdot)\|_{\var}=0\qquad\textrm{ for any }
x, y\in \R^d,\end{equation*} where $P_t(x,\cdot)$ is the transition
function of marginal process. A Markov coupling process is called
\emph{optimal} if it can give us sharp estimates of
$\|P_t(x,\cdot)-P_t(y,\cdot)\|_{\var}$ as $t$ tends to infinity. The
notion of optimal Markov coupling processes used here is different from
the one used by \cite[Definition 2.24]{CHEN}.

To construct an optimal Markov coupling process of subordinate
Brownian motion $(X_t)_{t\ge0}$, we begin with reviewing known facts
about the coupling of Brownian motions by reflection, see
\cite{LR,CL,HEP}. Fix $x$, $y\in\R^d$ with $x\neq y$. Let $B_t^x$ be
a Brownian motion on $\R^d$ $(d\ge1)$ starting from $x\in\R^d$, and
$H_{x,y}$ be the hyperplane such that the vector $x-y$ is normal
with respect to $H_{x,y}$ and $(x+y)/2\in H_{x,y}$, i.e.
$$H_{x,y}=\big\{u\in \R^d: \langle u- (x+y)/2, x-y\rangle=0\big\}.$$
Denote by $R_{x,y}: \R^d\rightarrow\R^d$ the reflection with
respect to the hyperplane $H_{x,y}$. Then, for every $z\in\R^d$,
$$R_{x,y}z=z-2\langle z-(x+y)/2, x-y\rangle(x-y)/|x-y|^2.$$ Define
$$\tau_{x,y}=\inf\big\{t>0: B^x_t\in H_{x,y}\big\}$$ and
$$\hat{B}_t^y:= \begin{cases} R_{x,y}B_t^x,& t\le\tau_{x,y};\\
\,\,\,\,B^x_t\,\,,& t>\tau_{x,y}.\end{cases}$$ That is,
$\hat{B}_t^y$ is the mirror reflection of $B_t^x$ with respect to
$H_{x,y}$
 before $\tau_{x,y}$ and coincides with $B_t^x$ afterwards. It is clear that $\hat{B}_t^y$ is a Brownian motion starting from $y$.
  Set $\tilde{B}^{x,y}_t:=(B_t^x,\hat{B}_t^y)$. Then,
  $\tilde{B}_t^{x,y}$ is a coupling of two Brownian motions starting from $x$, $y\in\R^{d}$ respectively. The coupling time
\begin{equation}\label{sub1}T_{x,y}^B:=\inf\{t>0: B_t^x=\hat{B}_t^y\}\end{equation} is just the stopping
time $\tau_{x,y}$. By \cite[Section 5, Page 170]{CL}, we have
\begin{equation}\label{sub2}\Pp(T_{x,y}^B>t)=\sqrt{\frac{2}{\pi}}\int_0^{|x-y|/(2\sqrt{2t})}\exp\big(\!\!-u^2/2\big)du\le
\frac{|x-y|}{2\sqrt{\pi t}}.\end{equation}
Note that $B_t$ here is just the usual standard Brownian motion but running at twice the speed, so the factor $\sqrt{2}$ appears in the upper bound of integral in \eqref{sub2}.

Next, let $(S_t)_{t\ge0}$ be a subordinator with $S_0=0$, which is
independent of $\tilde{B}_t^{x,y}$. Set
$$\tilde{X}^{x,y}_t=\tilde{B}^{x,y}_{S_t}=(B_{S_t}^x, \hat{B}_{S_t}^y).$$ Since $S_0=0$,
according to the definition of subordinate Brownian motion, we get
that $\tilde{X}^{x,y}_t$ is a coupling process of $X_t$ starting
from $x$ and $y$. For simplicity, let
$\tilde{X}^{x,y}_t:=(X_t^x,\hat{X}_t^{y}),$ and call
$\tilde{X}^{x,y}_t$ the \emph{reflection-subordinate coupling} of $X_t$.
Define the coupling time of $\tilde{X}^{x,y}_t$ as follows
\begin{equation}\label{sub3}T_{x,y}^X:=\inf\{t\ge0:X_t^x=\hat{X}_t^{y}\}.\end{equation} For any $x$, $y\in\R^d$, we will claim that $T_{x,y}^X<\infty$ almost surely. More precisely, we have

\begin{theorem}\label{th1} Let $X_t$ be a subordinate Brownian motion on $\R^d$ corresponding to the Bernstein function $f$, and
$P^f_t(x,\cdot)$ be its transition function. Then, $X_t$ has the coupling property; moreover, for any $t>0$ and $x$, $y\in\R^d$,
\begin{equation}\label{th11}\|P^f_t(x,\cdot)-P^f_t(y,\cdot)\|_{\var}\le 2\Pp(T_{x,y}^X>t)\le \frac{|x-y|}{\sqrt{2}\pi}\int_0^\infty \frac{1}{\sqrt{r}}e^{-tf(r)}dr.\end{equation}
 Additionally, assume that $\liminf\limits_{r\rightarrow\infty} f(r)/\log r>0$, $\liminf\limits_{r\rightarrow0}f(r)|\log r|<\infty$ and that $f^{-1}$ satisfies the following volume doubling property:
\begin{equation}\label{th1123}\limsup_{s\rightarrow0 }f^{-1}(2 s)/f^{-1}(s)<\infty.\end{equation}
Then, there exists a constant $C>0$ such that
 for $t>0$ sufficiently large \begin{equation}\label{th12}\|P^f_t(x,\cdot)-P^f_t(y,\cdot)\|_{\var}\le C|x-y|\sqrt{f^{^{-1}}\Big(\frac{1}{t}\Big)}.\end{equation}\end{theorem}

\begin{remark}\rm (i) We mention that \emph{if there exists $c>0$ such that for $s>0$ small enough, $2f(s)\le f(cs)$,
then \eqref{th1123} holds}.
Indeed, suppose that there exists $c_0>0$ such that $2f(s)\le f(cs)$ holds for
all $s\in (0,c_0]$. By the monotonicity of $f$, $f^{-1}(2f(s))\le cs$.
That is, $\limsup_{s\rightarrow0 }{f^{-1}(2f(s))}/{s}\le c$.
Since $f^{-1}(s)\rightarrow0$ as $s\rightarrow0$, we have
$\limsup_{s\rightarrow0 }{f^{-1}(2s)}/{f^{-1}(s)}\le c$, and so \eqref{th1123} follows.

(ii) It is clear that if
$\liminf\limits_{r\rightarrow\infty}{f(r)}/{\log r}>0$, then
$\int_0^\infty \frac{1}{\sqrt{r}}e^{-tf(r)}dr<\infty$ for $t>0$
large enough. We will claim that the converse is
also true. Indeed, assume that $\int_0^\infty
\frac{1}{\sqrt{r}}e^{-t_0f(r)}dr<\infty$. Since
$r\mapsto\frac{1}{\sqrt{r}}e^{-t_0f(r)}$ is strictly decreasing on
$[0,\infty)$, by a standard Abelian argument,
there exist positive constants $r_0$ and $c$ such that for any $r\ge
r_0$,
$$\frac{1}{\sqrt{r}}e^{-t_0f(r)}\le\frac{c}{r}.$$ That is,
${f(r)}/{\log r}\ge {c}/({2t_0}).$ So,
$\liminf\limits_{r\rightarrow\infty}{f(r)}/{\log r}\ge
{c}/({2t_0}).$
   \end{remark}

Before we prove Theorem \ref{th1} we give some examples. Here, we
always suppose that $S$ is a subordinator corresponding to the
Bernstein function $f$, and $X$ is the associated subordinate
Brownian motion. Denote by $P^f_t(x,\cdot)$ the transition function
of $X$. For two non-negative functions $g$ and $h$, the notation
$g\asymp h$ means that there are two positive constants $c_1$ and
$c_2$ such that $c_1g\le h\le c_2g.$ An extensive list of examples
of Bernstein functions can be found in \cite[Chapter 15]{RSW}.

\begin{example}\label{stable} Consider $\alpha\in(0,2)$ and define
$$f(\lambda)=\lambda^{\alpha/2}.$$ The corresponding subordinate Brownian
motion $X_t$ is the rotationally invariant stable L\'{e}vy process with
index $\alpha$. In this case, for $t>0$ sufficiently large, the estimate \eqref{th12} becomes
\begin{equation}\label{stable1}\|P^f_t(x,\cdot)-P^f_t(y,\cdot)\|_{\var}\le \frac{C|x-y|}{t^{1/\alpha}}.\end{equation}

On the other hand, let $Z_t$ be a
rotationally invariant $\alpha$-stable process on $\R^d$ starting
from $0$. For any $x$, $y\in\R^d$ with $x<y$, i.e.\ $x_i<y_i$ for $1\le i\le d$,
$$\aligned\|P^f_t(x,\cdot)-P^f_t(y,\cdot)\|_{\var}\ge& \big|\Pp\big(Z_t+x\in \Pi_{i=1}^d  (x_i,\infty)\big)-\Pp\big(Z_t+y\in \Pi_{i=1}^d(x_i,\infty)\big)\big|\\
=&\big|\Pp\big(Z_t\in (0,\infty)^d\big)-\Pp\big(Z_t\in \Pi_{i=1}^d(x_i-y_i,\infty)\big)\big|\\
=&\Pp\big(Z_t\in\Pi_{i=1}^d(x_i-y_i,\infty)\setminus(0,\infty)^d\big)\\
\ge &\sum_{j=1}^d\Pp\big(Z_t\in(x_j-y_j,0]\times
(0,\infty)^{d-1}\big).\endaligned$$ Denote by $p_t$ the density
function of $Z_t$. It is well known, see e.g.\ \cite{ChenK1, BSS}, that
$$p_t(z)\asymp t^{-d/\alpha}\wedge\frac{t}{|z|^{d+\alpha}}.$$ Thus,
for any $t\ge(y_1-x_1)^\alpha$,
$$\aligned &\int_{(x_1-y_1,0]\times (0,\infty)^{d-1}}p_t(z)dz\\\ge& c_0\int_{(x_1-y_1,0]\times (0,\infty)^{d-1}}\bigg(t^{-d/\alpha}\wedge\frac{t}{|z|^{d+\alpha}}\bigg)dz\\
\ge& c_0\int_{\sum_{i=2}^dz_i^2\ge t^{2/\alpha}, z_i>0,\, i=2,\cdots,d}\int_{x_1-y_1}^0\Bigg[\frac{t}{\big((y_1-x_1)^2+\sum_{i=2}^dz_i^2\big)^{(d+\alpha)/2}}\Bigg]dz_1dz_2\cdots dz_d\\
\ge& c_1(y_1-x_1)t\int_{\sum_{i=2}^dz_i^2\ge t^{2/\alpha}, z_i>0,\, i=2,\cdots,d}\frac{1}{\big(\sum_{i=2}^dz_i^2\big)^{(d+\alpha)/2}}dz_2\cdots dz_d\\
=&c_2(y_1-x_1)t\int_{t^{1/\alpha}}^\infty \frac{1}{r^{2+\alpha}}dr\\
=&\frac{c_2(y_1-x_1)}{t^{1/\alpha}}.\endaligned$$ Therefore, for $t\ge \max_{i=1}^d(y_i-x_i)^\alpha$,
$$\aligned\|P^f_t(x,\cdot)-P^f_t(y,\cdot)\|_{\var}\ge& \sum_{j=1}^d\int_{(x_j-y_j,0]\times (0,\infty)^{d-1}}p_t(z)dz\\
\ge&\frac{c_2\sum_{j=1}^d(y_j-x_j)}{t^{1/\alpha}}\\
\ge&\frac{c_2|y-x|}{t^{1/\alpha}}.\endaligned$$This implies that \eqref{stable1} is sharp.
\end{example}

\begin{example} Consider $0<\alpha<\beta<2$ and define
$$f(\lambda)=\lambda^{\alpha/2}+\lambda^{\beta/2}.$$ The corresponding subordinate Brownian
motion $X_t$ is a mixture of rotationally invariant stable L\'{e}vy processes
with index $\alpha$ and $\beta$. For this example, for $t>0$ large
enough,
$$\sqrt{f^{^{-1}}\Big(\frac{1}{t}\Big)}\asymp
\frac{1}{t^{1/\alpha}},$$ and so
$$\|P^f_t(x,\cdot)-P^f_t(y,\cdot)\|_{\var}\le \frac{C|x-y|}{t^{1/\alpha}}.$$ That is, the degree of
decay of $\|P^f_t(x,\cdot)-P^f_t(y,\cdot)\|_{\var}$ (as $t$
tends to infinity) is determined by the smaller index $\alpha.$ One
also can see this assertion in the following way: Let $P_t^{(\alpha)}$
and $P_t^{(\beta)}$ be the semigroups corresponding to subordinate
Brownian motions with Bernstein functions $f^{(\alpha)}(r)=r^{\alpha/2}$
and $f^{(\beta)}(r)=r^{\beta/2}$, respectively. According to the proof
of Proposition \ref{couooo} below, we have
$$\|P^f_t(x,\cdot)-P^f_t(y,\cdot)\|_{\var}\le\|P_t^{(\alpha)}(x,\cdot)-P_t^{(\alpha)}(y,\cdot)\|_{\var}\wedge\|P_t^{(\beta)}(x,\cdot)-P_t^{(\beta)}(y,\cdot)\|_{\var}.$$
Then, the desired assertion follows from Example \ref{stable} above.

Recently it has been proven in \cite[Theorem
1.2]{ChenK} (see also \cite{CKS}) that the density function of
a mixture of rotationally invariant stable L\'{e}vy processes with index
$\alpha$ and $\beta$ satisfies
$$p(t,x,y)\asymp\big(t^{-d/\alpha}\wedge t^{-d/\beta}\big)\wedge\bigg(\frac{t}{|x-y|^{d+\alpha}}+\frac{t}{|x-y|^{d+\beta}}\bigg)$$
 on $(0,\infty)\times \R^d \times\R^d$. Since $\alpha<\beta$, for $t>0$ large enough, $t^{-d/\alpha}<
 t^{-d/\beta}$, and so there exists $c>0$ such that for $t>0$ large
 enough,
$$p(t,x,y)\ge c\bigg(t^{-d/\alpha}\wedge\frac{t}{|x-y|^{d+\alpha}}\bigg).$$ The
right hand side of the inequality above is just the sharp
estimate (up to a constant) of the density function of rotationally
invariant $\alpha$-stable L\'{e}vy process. This implies that for this
example our upper bound $t^{-1/\alpha}$ is optimal for
$t>0$ large enough, cf.\ also Example \ref{stable}.
\end{example}

The following two Bernstein functions are taken from \cite[Chapter
5.2.2; Examples 2.15 and 2.16]{SONGZ}.

\begin{example} Consider $0<\alpha<2$, $\beta\in(0,2-\alpha)$ and define
$$f(\lambda)=\lambda^{\alpha/2}{(\log(1+\lambda))^{\beta/2}}.$$
Noting that $f(\lambda)\asymp \lambda^{(\alpha+\beta)/2}$ as
$\lambda\rightarrow 0$, for the corresponding subordinate Brownian motion,
$$\|P^f_t(x,\cdot)-P^f_t(y,\cdot)\|_{\var}\le \frac{C|x-y|}{t^{1/(\alpha+\beta)}}\qquad\textrm{ for }t>0 \textrm{ large enough}.$$
\end{example}

\smallskip

\begin{example} Consider $0<\alpha<2$, $\beta\in(0,\alpha)$ and define
$$f(\lambda)=\lambda^{\alpha/2}{(\log(1+\lambda))^{-\beta/2}}.$$ Since $f(\lambda)\asymp \lambda^{(\alpha-\beta)/2}$ as
$\lambda\rightarrow 0$, we know that in this situation, for $t>0$
large enough,
$$\|P^f_t(x,\cdot)-P^f_t(y,\cdot)\|_{\var}\le \frac{C|x-y|}{t^{1/(\alpha-\beta)}}.$$
\end{example}

As we can see from \eqref{th12} and, in particular, by the four examples from
above, estimates about $\|P^f_t(x,\cdot)-P^f_t(y,\cdot)\|_{\var}$ depend
on the decay of $f(\lambda)$ as $\lambda$ tends to zero. Roughly
speaking, the smaller $f(\lambda)$ near zero, the larger is the upper
bound in \eqref{th12}. The following three examples further
illustrate this point. They also show that under the
reflection-subordinate coupling the coupling happens not necessarily faster than under (compound) Poissonian coupling,
 cf.\ \eqref{1proff111} and \cite[Theorem 1.1]{RWJ}).

\begin{example}\label{ex11} Consider $0<\alpha<2$, $m>0$ and define
$$f(\lambda)=(\lambda+m^{2/\alpha})^{\alpha/2}-m.$$ We state that
as $\lambda\rightarrow 0$, $f(\lambda)\asymp \lambda$. The
corresponding subordinate process is the relativistic stable L\'{e}vy process.
For $t>0$ large enough,
$$\|P^f_t(x,\cdot)-P^f_t(y,\cdot)\|_{\var}\le \frac{C|x-y|}{\sqrt{(m+t^{-1})^{2/\alpha}-m^{2/\alpha}}}\asymp\frac{C|x-y|}{\sqrt{t}}.$$
The estimate above is sharp for $t>0$ large enough. Indeed, for
$m=\alpha=d=1$, it can be shown that (e.g.\ see \cite{HS, CM} or
\cite[Example 2.4]{ChenK}) for every $t>0$ and $(x,y)\in \R \times
\R$, the density function of relativistic stable L\'{e}vy process satisfies
$$p(t,x,y)\ge \frac{c_1t}{(|x-y|+t)^{2}}\bigg(1\vee (|x-y|+t)^{1/2}\bigg)e^{-c_2\frac{|x-y|^2}{\sqrt{|x-y|^2+t^2}}}.$$ In particular, for $t>0$ large enough, we have
$$p(t,x,y)\ge \frac{c_1t}{(|x-y|+t)^{3/2}}e^{-c_2\frac{|x-y|^2}{\sqrt{|x-y|^2+t^2}}}.$$
Let $Z_t$ be a relativistic stable L\'{e}vy process with $m=\alpha=1$ on $\R$
starting from $0$. Then, for any $x$, $y\in\R$ with $x<y$ and $t>0$ large enough,
$$\aligned\|P^f_t(x,\cdot)-P^f_t(y,\cdot)\|_{\var}\ge& \big|\Pp\big(Z_t+x\in (x,\infty)\big)-\Pp\big(Z_t+y\in (x,\infty)\big)\big|\\
=& \big|\Pp\big(Z_t\in (0,\infty)\big)-\Pp\big(Z_t\in (x-y,\infty)\big)\big|\\
=&\Pp\big(Z_t\in (x-y,0]\big)=\int_{x-y}^0p(t,0,z)dz\\
\ge&c_1t\int_{x-y}^0
\frac{1}{(|u|+t)^{{3/2}}}e^{-c_2\frac{u^{2}}{\sqrt{u^2+t^2}}}du\\
\ge&
\frac{C_1(y-x)}{t^{1/2}}.\endaligned$$
\end{example}

\smallskip

\begin{example}\label{ex12} First, we consider $0<\alpha\le1$ and define
$$f(\lambda)=\log^{1/\alpha}(1+\lambda^\alpha),$$ which satisfies that $f(\lambda)\asymp \lambda$ as
$\lambda\rightarrow 0$. When $\alpha=1$, $S_t$ is called Gamma
subordinator. In this setting, for $t>0$ large enough,
$$\|P^f_t(x,\cdot)-P^f_t(y,\cdot)\|_{\var}\le
{C|x-y|}\Big(\exp({t^{-\alpha}})-1\Big)^{1/(2\alpha)}\asymp\frac{C|x-y|}{\sqrt{t}}.$$

On the other hand, we study the coupling property of rotationally
invariant geometric stable L\'{e}vy processes, which are subordinate
Brownian motions associated with the Bernstein function
$g(\lambda)=\log(1+\lambda^\alpha)$ and $0<\alpha\le 2$. One can see
that for these processes, when $t>0$ large enough,
$$\|P^g_t(x,\cdot)-P^g_t(y,\cdot)\|_{\var}\le \frac{C|x-y|}{t^{1/\alpha}}.$$
This assertion is the same as that for rotationally invariant stable
L\'{e}vy processes, but completely different from Brownian motions
subordinated with $f$. We furthermore point out that for
rotationally invariant geometric stable L\'{e}vy processes,
$g(\lambda)\asymp \lambda^\alpha$ as $\lambda\rightarrow 0$.
\end{example}

\smallskip

Now, we turn to the proof of Theorem \ref{th1}.

\begin{proof} [Proof of Theorem \ref{th1}] We follow the proof of \cite[Proposition
3.3]{RWJ} to verify that the relation between coupling times defined
by \eqref{sub1} and \eqref{sub3} is
\begin{equation}\label{sub4}T_{x,y}^X=\inf\big\{t\ge0:S_t\ge T_{x,y}^B\big\}.\end{equation}
In the following argument, assume that $\omega$ is fixed, and define
$K_{x,y}=\inf\big\{t\ge0:S_t\ge T_{x,y}^B\big\}.$ Let $t>0$ be such
that $S_t\ge T_{x,y}^B$, i.e.\ $t\geq K_{x,y}$. Since
$B^x_t=\hat{B}^y_t$ for $t\ge T_{x,y}^B$, we have
$B^x_{S_t}=\hat{B}^y_{S_t}$, and, by construction,
$X_t=\hat{X}_t$. That is, $T_{x,y}^X\le t$. Since $t\geq K_{x,y}$
was arbitrary, we have $T^X_{x,y}\le K_{x,y}$. On the other hand,
assume that $K_{x,y}>0$. Then, by the definition of $K_{x,y}$, for
any $\varepsilon>0$, there exists $t_\varepsilon>0$ such that
$t_\varepsilon>K_{x,y}-\varepsilon$ and $S_{t_\varepsilon}<
T_{x,y}^B$. Hence, $B^x_{S_{t_\varepsilon}}\neq
\hat{B}^y_{S_{t_\varepsilon}}$, i.e.\ $X^x_{t_\varepsilon}\neq
\hat{X}^y_{t_\varepsilon}$. Therefore, $T_{x,y}^X\ge
t_\varepsilon>K_{x,y}-\varepsilon$. Letting
    $\varepsilon\rightarrow0$, we get $T_{x,y}^X\ge K_{x,y}$, and thus \eqref{sub4} holds.

Now, according to \eqref{sub2}, for almost every $\omega$ we have
$T_{x,y}^B(\omega)<\infty$. Since the subordinator $S_t$ tends to
infinity as $t\to\infty$, there exists $\tau_0(\omega)<\infty$ such
that $S_{t}(\omega)\ge T_{x,y}^B(\omega)$ for all $t\geq\tau_0(\omega)$.
Therefore, \eqref{sub4} implies that $T_{x,y}^X\leq \tau_0 <\infty$.

\medskip

For any $x$, $y\in\R^d$ and $t>0$, by the classic coupling
inequality, \eqref{sub4} and \eqref{sub2},
$$\aligned \|P_t(x,\cdot)-P_t(y,\cdot)\|_{\var}
&\le2\Pp(T_{x,y}^X>t)\\
&=2\Pp(T^B_{x,y}>S_t)\\
&=2\int_0^\infty\Pp(T^B_{x,y}>s)\Pp(S_t\in
ds)\\
&\le \frac{|x-y|}{\sqrt{\pi}}\int_0^\infty\frac{1}{\sqrt{s}}\,\mu_t^S(ds).\endaligned$$
According to the fact that
$$\frac{1}{\sqrt{s}}=\frac{1}{\sqrt{2\pi}}\int_0^\infty
\frac{1}{\sqrt{r}}e^{-rs}dr,$$ we obtain
$$\int_0^\infty\frac{1}{\sqrt{s}}\,\mu_t^S(ds)=\frac{1}{\sqrt{2\pi}}\int_0^\infty\frac{1}{\sqrt{r}}
\int_0^\infty
e^{-rs}\mu_t^S(ds)dr=\frac{1}{\sqrt{2\pi}}\int_0^\infty\frac{1}{\sqrt{r}}e^{-tf(r)}dr,$$
which in turn gives us \eqref{th11}.

\medskip

Since the Bernstein function $f$ is strictly increasing, we can make a change of
variables to get
$$\int_0^\infty \frac{1}{\sqrt{r}}e^{-tf(r)}dr=\int_0^\infty
\frac{e^{-ts}}{\sqrt{f^{-1}(s})f'(f^{-1}(s))}ds=2\int_0^\infty
e^{-ts}d\sqrt{f^{-1}(s)}.$$ Suppose that $\liminf\limits_{r\rightarrow\infty} f(r)/\log r>0$ and \eqref{th1123} hold. Then, we can choose positive constants $c_i$ $(i=1,2,3)$ such that $f^{-1}(2x)\le c_1f^{-1}(x)$ if $x\in (0,2c_3]$; $f^{-1}(x)\le e^{c_2x}$ if $x\in [c_3, \infty)$. Thus, for $t>0$ large enough,
$$\aligned \int_0^\infty
e^{-ts}d\sqrt{f^{-1}(s)}&=e^{-ts}\sqrt{f^{-1}(s)}\,\bigg|_0^\infty+\int_0^\infty f^{-1}\Big(\frac{s}{t}\Big)e^{-s}ds\\
&=\int_0^\infty f^{-1}\Big(\frac{s}{t}\Big)e^{-s}ds.\endaligned$$  For any $s\in(1,c_3t]$, choose $k=[\log_2 s]+1.$ Since $f^{-1}$ is increasing, we find
$$f^{-1}\Big(\frac{s}{t}\Big)\le f^{-1}\Big(\frac{2^k}{t}\Big)\le c_1^kf^{-1}\Big(\frac{1}{t}\Big)\le 2^{k\rho} f^{-1}\Big(\frac{1}{t}\Big)\le(2s)^\rho f^{-1}\Big(\frac{1}{t}\Big),$$ where $\rho=\log_2 c_1$. Therefore,  for $t>0$ large enough,
$$\aligned \int_0^\infty f^{-1}\Big(\frac{s}{t}\Big)e^{-s}ds&=\bigg(\int_0^1+\int_1^{c_3t}+\int_{c_3t}^\infty\bigg)f^{-1}\Big(\frac{s}{t}\Big)e^{-s}ds\\
&\le f^{-1}\Big(\frac{1}{t}\Big)+2^\rho f^{-1}\Big(\frac{s}{t}\Big)\int_1^{c_3t} s^\rho e^{-s}ds+\int_{c_3 t}^\infty e^{-(s-c_2 s/t)}ds\\
&\le \bigg[1+2^\rho\int_1^\infty s^\rho e^{-s}ds\bigg]f^{-1}\Big(\frac{1}{t}\Big)+\int_{c_3t}^\infty e^{-s/2}ds\\
&\le C_1 f^{-1}\Big(\frac{1}{t}\Big) +2e^{-c_3t/2}.\endaligned$$ Since $\liminf_{r\rightarrow0} f(r)|\log r|<\infty$, there exist positive constants $c_4$ and $r_0$ such that for $r\le r_0$, $f(r)\le c_4\big/\log r^{-1}.$ Thus, for $t>0$ large enough, $f^{-1}(1/t)\ge \exp(-c_4t)$. According to the volume doubling property \eqref{th1123} again, we get that for $t>0$ large enough, $$f^{-1}(1/t)\ge c_5e^{-c_3t/2}.$$ This along with all the above conclusions above yields the required assertion.
 \end{proof}

Theorem \ref{th1} is easily generalized to study the coupling
property of L\'{e}vy processes, which can be decomposed into two
independent parts, one of which is a subordinate Brownian motion.
\begin{proposition}\label{couooo} Suppose that the L\'{e}vy process $X_t$ can be split into $$X_t=Y_t+{B^f_t},$$ where
$B^f_t$ is a Brownian motion subordinated by the subordinator
$S$ and $Y_t$ is a L\'{e}vy process. Let $P_t(x,\cdot)$ be the transition probability function of
$X_t$. Then, there exists a constant $C>0$ such that for $t>0$ and
$x$, $y\in\R^d$,
\begin{equation*}\|P_t(x,\cdot)-P_t(y,\cdot)\|_{\var}\le \bigg(\frac{|x-y|}{\sqrt{2}\pi}\int_0^\infty \frac{1}{\sqrt{r}}e^{-tf(r)}dr\bigg)\wedge \frac{C(1+|x-y|)}{\sqrt{t}}\wedge 2,\end{equation*}
where $f(\lambda)$ is the Bernstein function corresponding to $S$.
\end{proposition}
\begin{proof} Let $P_t^{f}$ and $P^Y_t$ be the semigroups of ${B^f_t}$ and $Y_t$ respectively. Then,
\begin{align*}%\label{pth20}
    \|P_t(x,\cdot)-P_t(y,\cdot)\|_{\var}
    &= \sup_{\|f\|_\infty\le 1}\big|P_tf(x)-P_tf(y)\big|\\
    &=\sup_{\|f\|_\infty\le 1} \big|P_t^{f}P_t^Yf(x)-P_t^{f} P_t^Yf(y)\big|\\
    &\le \sup_{\|h\|_\infty\le 1} \big|P^{f}_th(x)-P^{f}_th(y)\big|\\
    &= \|P^{f}_t(x,\cdot)-P^{f}_t(y,\cdot)\|_{\var}.
\end{align*} Note that the L\'{e}vy measure of any subordinate
Brownian motion is absolutely continuous with respect to the Lebesgue measure.
According to \cite[Theorem 4.3 and Corollary 4.4]{RWJ}, there exists
a constant $C>0$ such that for any $t>0$, $x$, $y\in\R^d$,
$$\|P^{f}_t(x,\cdot)-P^{f}_t(y,\cdot)\|_{\var}\le
\frac{C(1+|x-y|)}{\sqrt{t}}\wedge 2.$$ Combining this with Theorem
\ref{th1} yields
\begin{equation*}\|P^{f}_t(x,\cdot)-P^{f}_t(y,\cdot)\|_{\var}\le\bigg(\frac{|x-y|}{\sqrt{2}\pi}\int_0^\infty \frac{1}{\sqrt{r}}e^{-tf(r)}dr\bigg)\wedge \frac{C(1+|x-y|)}{\sqrt{t}}\wedge 2.\end{equation*}
We have proved the desired assertion.\end{proof}

\begin{remark}\rm
Proposition \ref{couooo} can be stated in the following way: \emph{Let $\Phi(\xi)$ be the symbol of L\'{e}vy process $X_t$.
If $\Phi(\xi)= f(|\xi|^2)+\Psi(\xi)$, where $f$ is a
Bernstein function and $\Psi(\xi)$ is again a symbol of a L\'{e}vy
process, then the conclusion of Proposition \ref{couooo} holds.}\end{remark}

\section{Proof and Extension of Theorem \ref{section1th1}}\label{section3}
We begin with the proof of Theorem \ref{section1th1}.
\begin{proof}[Proof of Theorem \ref{section1th1}] We first suppose that the L\'{e}vy measure of $X_t$ satisfies
\begin{equation}\label{proof2233}\nu(dz)\ge c|z|^{-d}f(|z|^{-2})dz,\end{equation} where $$c=\bigg(\int_{\{|z|\le 1\}}(1-\cos z_1)|z|^{-d}dz\bigg)^{-1}.$$ Then, by the argument of \cite[Theorem 1.1, Part (a)]{W2},
$$X_t=X_t'+{B_t^{f}},$$ where ${B_t^{f}}$ is a
subordinated Brownian motion corresponding to the Bernstein function
$f$ and $X_t'$
is a L\'{e}vy process with symbol  $$
    \psi'(\xi)= \frac{1}{2}\langle Q\xi, \xi\rangle +i\langle b,\xi\rangle + \int_{z\neq 0}\!\! \Big( 1-e^{-i\langle \xi,z\rangle}+i\langle \xi,z\rangle\I_{\{|z|\le 1\}}\Big)\nu_{X'}(\xi,dz),
$$ where
$$\nu_{X'}(\xi,dz):=\nu(dz)- c|z|^{-d}f(|\xi|^{2})\I _{\{|z|\le |\xi|^{-1}\}}dz\ge0.$$ Therefore, Theorem
\ref{section1th1} is a consequence of Proposition \ref{couooo}.

 Next, we turn to consider the condition \eqref{1proff2}. Since $$1-\cos z_1\ge \frac{\cos 1}{2}z_1^2\qquad\textrm{ for }|z|\le 1,$$ we have $$\int_{|z|\le 1}(1-\cos z_1)|z|^{-d}dz\ge\frac{\cos 1}{2}\int_{|z|\le 1}|z_1|^2|z|^{-d}dz.$$ By symmetry, for $i=1, \cdots, d$,
$$\int_{|z|\le 1}|z_1|^2|z|^{-d}dz=\int_{|z|\le 1}|z_i|^2|z|^{-d}dz=\frac{1}{d}\int_{|z|\le 1}|z|^{-d+2}dz=\frac{c_d}{d},$$ where $c_d=\pi^{d/2}/\Gamma(d/2+1)$, i.e.\ the volume of the unit ball in $\R^d$. Therefore,
 $$\int_{|z|\le 1}(1-\cos z_1)|z|^{-d}dz\ge\frac{c_d\cos1}{2d}.$$ That is, $c\le {2d}/({c_d\cos1}).$ This combining with \eqref{1proff2} and \eqref{proof2233} gives us the required conclusion.
\end{proof}

Having Theorem \ref{section1th1} in mind, the following condition seems to be more natural: the L\'{e}vy measure $\nu$ has \emph{only around the
origin} an absolutely continuous component, i.e. there exists $r\in(0,\infty]$ such that
\begin{equation}\label{1proff211}\nu(dz)\ge |z|^{-d}f(|z|^{-2})\I_{\{|z|\le r\}}dz,\end{equation} where
$f$ is a Bernstein function. A similar lower bound condition has
already been used in \cite{W2} to study gradient estimates for
Ornstein-Ohlenbeck jump processes. According to \cite[Corollary
4.1]{RWJ}, cf.\ also the remark below Theorem \ref{section1th1}, we
know that under condition \eqref{1proff211}, the associated L\'{e}vy
process $X_t$ has the coupling property and \eqref{1proff111} holds.
However, the following example shows that assertion \eqref{1proff1}
is not satisfied.

\begin{example} Consider the truncated rotationally invariant stable L\'{e}vy process $X_t$ on $\R$ with
index $\alpha$. The corresponding L\'{e}vy measure is given by
$$\nu(dz)=\frac{c_{\alpha}}{|x|^{1+\alpha}}\I_{\{|z|\le 1\}},$$ where $c_{\alpha}$ is a constant depending only on $\alpha$.
Then, for any $x$, $y\in\R$ and $t>0$,
$$\|P_t(x,\cdot)-P_t(y,\cdot)\|_{\var}\asymp \frac{1}{\sqrt{t}}.$$ Indeed, on the one hand, by the remark below \eqref{1proff211}, there exists some $C_1>0$ such that for any $x$, $y\in\R$ and $t>0$, we have
$$\|P_t(x,\cdot)-P_t(y,\cdot)\|_{\var} \le\frac{C_1(1+|x-y|)}{\sqrt{t}}.$$ On the other hand, let $p_t(x,y)$ be the transition density function of $X_t$. According to \cite[Theorem 3.6]{CKTJ}, there exist $c_0$, $c_1$, $c_2$, $c_3$ and $c_4>0$ such that
\begin{equation}\label{WWWEEW}
    p_t(x,y)
   \ge
    \begin{cases}
        c_0t^{-1/2}& t\ge R_*^\alpha,\, |x-y|^2\le t;\\
       c_1\bigg(\frac{t}{|x-y|}\bigg)^{c_2|x-y|},& |x-y|\ge \max\{t/C_*, R_*\};\\
       c_3t^{-1/2}\exp\bigg(-\frac{c_4|x-y|^2}{t}\bigg),& C_*|x-y|\le t\le |x-y|^2,
    \end{cases}
\end{equation} where $R_*$ and $C_*$ are two positive constants.
Denote by $Z_t$ a truncated rotationally invariant stable L\'{e}vy process on $\R$
starting from $0$. Then, for any $x$, $y\in\R$ with $x<y$ and $t\ge |x-y|^2\wedge R_*^\alpha$,
$$\aligned\|P_t(x,\cdot)-P_t(y,\cdot)\|_{\var}\ge&\big|\Pp\big(Z_t+x\in (x,\infty)\big)-\Pp\big(Z_t+y\in (x,\infty)\big)\big|\\
=& \big|\Pp\big(Z_t\in (0,\infty)\big)-\Pp\big(Z_t\in (x-y,\infty)\big)\big|\\
=&\Pp\big(Z_t\in (x-y,0]\big)=\int_{x-y}^0p_t(0,z)dz\\
\ge&
\frac{c_0(y-x)}{t^{1/2}},\endaligned$$ where in the last inequality we have used \eqref{WWWEEW}. The required assertion follows.
  \end{example}

The following result is an analog of Theorem \ref{section1th1}.

\begin{theorem}\label{llsection1th1} Let $X_t$ be a L\'{e}vy process on $\R^d$ and $\nu$ be its L\'{e}vy measure. Assume that
\begin{equation}\label{ll1proff2}\nu(dz)\ge \sum_{i=1}^d\bigg(|z_i|^{-1}f_i(|z_i|^{-2})\I_{\{z_1=\cdots= z_{i-1}=z_{i+1}=\cdots= z_{d}=0\}}\bigg)dz,\end{equation} where the
$f_i$ are Bernstein functions. Then, there exists $C>0$ such that
for any $x$, $y\in\R^d$ and $t>0$,
$$\|P_t(x,\cdot)-P_t(y,\cdot)\|_{\var}\le 2\wedge
\sum_{i=1}^d\bigg[\Big(\frac{|x_i-y_i|}{\sqrt{2}\pi}\int_0^\infty
\frac{1}{\sqrt{r}}e^{-ctf_i(r)}dr\Big)\wedge
\frac{C(1+|x_i-y_i|)}{\sqrt{t}}\bigg],$$ where $c=\pi^{d/2}\cos
1\big/ (2d\Gamma(d/2+1))$.

\end{theorem}

Unlike \eqref{1proff2} in Theorem \ref{section1th1},
\eqref{ll1proff2} is a condition on the L\'{e}vy measure restricted
on the coordinate axes. Here, we mention one significant example
which satisfies \eqref{ll1proff2} (but not \eqref{1proff2}).

\begin{example}\rm Set $L=(L^{(1)},
\cdots,L^{(d)})$, where $L^{(1)},$  $L^{(2)},$ $\cdots$ $L^{(d)}$
are independent L\'{e}vy processes on $\R$. The L\'{e}vy
measure $\nu$ of $L$ is concentrated on the coordinate
axes. Assume that $\nu$ has the following density
$$\sum_{i=1}^d\Bigg(\I_{\{z_1=\cdots= z_{i-1}=z_{i+1}=\cdots=
z_{d}=0\}}\frac{c_i}{|z|_i^{1+\alpha}}\Bigg)dz,$$ where $c_i>0$
$(1\le i\le d)$ are constants. (Note that this measure is
\textit{more} singular than the standard rotationally invariant
$\alpha$-stable L\'{e}vy process.) Then, there exists a constant
$C>0$ such that for any $t>0$,
$$\|P_t(x,\cdot)-P_t(y,\cdot)\|_{\var}\le \frac{C|x-y|}{t^{1/\alpha}},$$ where $P_t(x,\cdot)$ is the transition function of $L$.\end{example}

\begin{proof} [Proof of Theorem \ref{llsection1th1}] Under condition \eqref{ll1proff2}, we can split the L\'{e}vy process $X_t$ into
$$X_t=Y_t+Z_t,$$ where $Y_t$ is a pure L\'{e}vy jump process with L\'{e}vy measure
$$\nu_{Y}(dz):= \sum_{i=1}^d\bigg(|z_i|^{-1}f_i(|z_i|^{-2})\I_{\{z_1=\cdots= z_{i-1}=z_{i+1}=\cdots= z_{d}=0\}}\bigg)dz.$$ $Z_t$ is independent of $Y_t$, and it has the
L\'{e}vy measure
$$\nu_{Z}(dz):=\nu(dz)- \sum_{i=1}^d\bigg(|z_i|^{-1}f_i(|z_i|^{-2})\I_{\{z_1=\cdots= z_{i-1}=z_{i+1}=\cdots= z_{d}=0\}}\bigg)dz\ge0.$$ According to the definition of $\nu_{Y}$, the generator
of $Y$ is
$$L_{Y}h(x)=\sum_{i=1}^d\int_{\R}\!\!\Big(h(x+ue_i)-h(x)-\I_{_{\{|u|\le 1\}}}u\partial_{x_i}h(x)\Big)|u|^{-1}f_i(|u|^{-2})du,$$ where $h\in C_b^2(\R^d)$ and $e_i$ is the canonical basis
in $\R^d$. Therefore,
$$Y_t=(L^{(1)},
\cdots,L^{(d)}),$$ where $L^{(1)},$  $L^{(2)},$ $\cdots$ $L^{(d)}$
are independent one-dimensional L\'{e}vy processes with L\'{e}vy
measures
$$\nu_{L^{(i)}}(du):=|u|^{-1}f_i(|u|^{-2})du,\quad i=1,\cdots, d,$$ respectively.

Following the proofs of Theorems \ref{section1th1} and \ref{th1},
for $1\le i\le d$, there exists a coupling $(L^{(i)}, L^{'(i)})$ of
$L^{(i)}$ such that the coupling time $T^{(i)}_{x_i,y_i}$ (starting
from $x_i$ and $y_i$) satisfies
$$\Pp(T^{(i)}_{x_i,y_i}>t)\le \Big(\frac{|x_i-y_i|}{2\sqrt{2}\pi}\int_0^\infty
\frac{1}{\sqrt{r}}e^{-ctf_i(r)}dr\Big)\wedge
\frac{C(1+|x_i-y_i|)}{\sqrt{t}}$$ for some constant $C>0$ (which can
be chosen independently of $i$). In particular, the part $(Y, Y')$
with $Y'=(L^{'(1)}, \cdots,L^{'(d)})$ is a coupling of $Y$. Denote
by $T_{x,y}$ the coupling time of $(Y, Y')$. Then, due to the
independence of $L^{(1)},$ $L^{(2)},$ $\cdots$ $L^{(d)}$, we find
that (see \cite[Decomposition Lemma 4.18]{CL})
$$T_{x,y}=\max_{1\le i\le d}T^{(i)}_{x_i,y_i}.$$
Let $P^{Y}_t$ be the semigroup of $Y$. Therefore, for any $x$,
$y\in\R^d$,
$$\aligned \|P_t^{Y}(x,\cdot)-P_t^{Y}(y,\cdot)\|_{\var}\le &2\Pp\big(\max_{1\le i\le d}T^{(i)}_{x_i,y_i}>t\big)\\
\le&
\sum_{i=1}^d\bigg[\Big(\frac{|x_i-y_i|}{\sqrt{2}\pi}\int_0^\infty
\frac{1}{\sqrt{r}}e^{-ctf_i(r)}dr\Big)\wedge
\frac{2C(1+|x_i-y_i|)}{\sqrt{t}}\bigg],\endaligned$$ where
$x=(x_1,x_2,\cdots,x_d)$ and $y=(y_1,y_2,\cdots,y_d).$ Since
$$\|P_t(x,\cdot)-P_t(y,\cdot)\|_{\var}\le\|P_t^{X'}(x,\cdot)-P_t^{X'}(y,\cdot)\|_{\var},$$ as in the proof
of Proposition \ref{couooo}, we are done.
\end{proof}

We finally turn to the proof of Proposition \ref{propo}.

\begin{proof}[Proof of Proposition \ref{propo}] We assume that $c=1$
for simplicity. Let $(S_t)_{t\ge0}$ be a subordinator associated
with the Bernstein function $f$. For any $t\ge0$, let $\mu_t^S$ be
the transition probabilities of the subordinator $S$, i.e.
$\mu_t^S(B)=\Pp(S_t\in B)$ for any $B\in\Bb([0,\infty))$, and
$\Ee^S$ be its expectation. Then,
$$\aligned \int_0^\infty\frac{1}{\sqrt{s}}e^{-t f(s)}ds&=\int_0^\infty\frac{1}{\sqrt{s}}\int_0^\infty e^{-sr}\mu^S_t(dr)ds\\
&=\int_0^\infty\bigg(\int_0^\infty\frac{1}{\sqrt{s}}e^{-sr}ds\bigg)\mu^S_t(dr)\\
&=\int_0^\infty \frac{e^{-u}}{\sqrt{u}}du\int_0^\infty\frac{1}{\sqrt{r}}\mu^S_t(dr)\\
&=\Ee^S \bigg(\frac{\sqrt{2\pi}}{\sqrt{S_t}}\bigg).\endaligned$$
By using the Cauchy-Schwarz inequality twice, we arrive at
$$\aligned \int_0^\infty\frac{1}{\sqrt{s}}e^{-t f(s)}ds&\ge\frac{\sqrt{2\pi}}{\Ee^S \sqrt{S_t}}\ge \frac{\sqrt{2\pi}}{\sqrt{\Ee^S S_t}} .\endaligned$$
Since $S_t$ is a L\'{e}vy process starting from $0$, we easily see that
$\Ee^S S_t= t \Ee^S S_1$, which yields that
\begin{equation}\label{ppropo2} \int_0^\infty\frac{1}{\sqrt{s}}e^{-t f(s)}ds \ge \frac{\sqrt{2\pi}}{\sqrt{t\Ee^S S_1}}. \end{equation}

We claim that $\Ee^S S_1<\infty$ if and only if $f'(0+)<\infty.$ In fact, for any $x>0$, $\Ee^S e^{-xS_1}=e^{-f(x)}$. Then, $\Ee^S\big( S_1e^{-xS_1}\big)=f'(x)e^{-f(x)}.$ Letting $x\rightarrow0$, by the monotone convergence theorem and the definition of Bernstein function $f$, we have
$\Ee^S S_1=f'(0+).$ The desired assertion follows. Therefore, if $f'(0+)<\infty$, then, due to \eqref{ppropo2}, there exists a finite constant $C_1>0$ such that \begin{equation}\label{ppropo3}\int_0^\infty\frac{1}{\sqrt{s}}e^{-t f(s)}ds \ge \frac{C_1}{\sqrt{t}}.\end{equation}
The proof is completed by \eqref{ppropo3} and Theorem \ref{section1th1}.
\end{proof}

\bigskip

\noindent{\bf Acknowledgement.} Financial support through DFG (grant
Schi 419/5-1) and DAAD (PPP Kroatien) (for Ren\'{e} L.\ Schilling)
and the Alexander-von-Humboldt Foundation and the Natural Science
Foundation of Fujian $($No.\ 2010J05002$)$ (for Jian Wang) is
gratefully acknowledged. The authors would like to thank Prof.\ Z.
Vondra\v{c}ek and the referee for a number of corrections and
comments on earlier version of the paper.

\end{document}